\documentclass{amsart}
\usepackage{amsmath}%
\usepackage{amsfonts}%
\usepackage{amssymb}%
\usepackage{graphicx}
\usepackage{booktabs} 
\usepackage{tikz} 
\usepackage{tkz-graph}
\usepackage{tkz-berge}
\usepackage{float}
\usepackage{caption}
\usepackage{subcaption}
\usetikzlibrary{shapes,arrows}
\newtheorem{theorem}{Theorem}
\theoremstyle{plain}

\theoremstyle{definition}

\theoremstyle{plain}

\newtheorem{conjecture}{Conjecture}

\theoremstyle{definition}
\newtheorem{definition}{Definition}

\newtheorem{question}{Question}

\newtheorem{remark}{Remark}
\theoremstyle{plain}

\numberwithin{equation}{section}
\begin{document}
\title[Swapping \& Reconstruction]{2-Swappability and the Edge Reconstruction Number of Regular Graphs}
\author{Michael S. Ross}


\begin{abstract}
The edge-reconstruction number of graph $G$, denoted $ern(G)$,
	is the size of the smallest multiset of edge-deleted, unlabeled subgraphs
	of $G$, from which the structure of $G$ can be uniquely determined. 
	That there was some connection between the areas of edge 
	reconstruction and swappability	has been known since the swapping number of a graph was
	first introduced by Froncek, Rosenberg, and Hlavacek in 2013\cite{edgeswap}. 
	This paper illustrates the depth of that connection by proving several 
	bridging results	between those areas; in particular,
	when the graphs in question are both regular and 2-swappable.
	These connections led to the discovery of four infinite families of 
	$r\geq 3$ regular graphs with $ern(G) \geq 3$, contradicting the
	formerly conjectured upper bound.
\end{abstract}
\maketitle

\section*{Background}
We begin with some basic definitions in the areas of both graph reconstruction and the swappability of graphs. For the extent of this writing, all graphs introduced are understood to be finite,  simple, and connected unless explicitely stated otherwise.

\begin{definition}[Edge-Deck]
Let $G$ be a graph, and let $e\in E(G)$. Then the unlabeled graph $G-e$ is said to be an \textbf{edge-card} of $G$. The multiset of all such edge-cards of $G$ is called the \textbf{edge-deck} of $G$, and is written $\mathcal{ED}(G)$.
\end{definition}

\begin{definition}
Let $S\subseteq \mathcal{ED}(G)$. A \textbf{blocker} of $S$ is a graph $H\not\cong G$ such that $S\subseteq \mathcal{ED}(H)$.
If no such graph $H$ exists, then $G$ is said to be \textbf{reconstructable} from $S$.
\end{definition}

\begin{definition}
Let $G$ be a graph such that $G$ is reconstructable from $\mathcal{ED}(G)$. Then the \textbf{edge reconstruction number} of $G$, noted $ern(G)$,
is the size of the smallest $S\subseteq \mathcal{ED}(G)$ such that $G$ is reconstructable from $S$.
\end{definition}

Froncek, Hlavacek, and Rosenberg \cite{edgeswap} define $k$-swappability as follows.

\begin{definition}Let $G$ be a graph and let $k$ be a positive integer.
We say that $G$ is {\bf $k$-swappable} if $\forall e\in E(G)$,  there exists $A \subseteq E(G)$ and  $B \subseteq E(\bar{G})$ such that $e \in A$, $|A| \leq k$, and 
$G \cong G - A + B$.

The smallest such $k$ is called the {\bf swapping number} of $G$. If no such number exists, the swapping number of $G$ is said to be infinity.
\end{definition}

It is worth noting that in the context of regular graphs, there can be no individually replaceable edges. Thus, if a regular graph is finitely swappable, the size of the swapping sets for any edge is at least 2. 

\newpage

\section*{Results}

Originally published in K.J. Asciak's master's thesis at the University of Malta in 1999, the following conjecture was republished in a recent (2010) survey regarding open questions in reconstruction numbers \cite{survey}.

\begin{conjecture}\label{conj}
If $G$ is an $r\geq 3$ regular graph, then $ern(G)\leq 2$.
\end{conjecture}

While exploring the possibility of this conjecture, we found the following relation between 2-swappability and $ern(G)$.

\begin{theorem}
If $G$ is a graph with $ern(G)\geq 3$, then $G$ is $2$-swappable.
\end{theorem}

\begin{proof}
Let $G$ be a graph, and suppose $ern(G)\geq 3$. Then $\forall S \subseteq \mathcal{ED}(G)$, where $\left|S\right|=2$, there is a graph that blocks $S$.
Let $e\in E(G)$, and let $f\in E(G-e)$. Then there must exist some blocker $H$ such that $\{G-e,G-f\}\subseteq \mathcal{ED}(H)$.  Therefore, there exist distinct $ \tilde{e},\tilde{f} \in E(H)$ such that: 

\begin{equation}H-\tilde{e}\cong G-e
\label{eq1}
\end{equation} 

\begin{equation}H-\tilde{f}\cong G-f.
\label{eq2}
\end{equation}
 This implies there is some $e'\in E(\bar{G})$ so that $G-e+e'\cong H$. By substituting the left hand side of this expression for $H$ in \eqref{eq2} we see that WLOG there must be some $f' \in E(G-e)$  so that
$G-e+e'-f'\cong G-f$ [See remark below], and thus there must exist $f''\in E(\bar{G})$ so that $G-e+e'-f'+f''\cong G$. Since $\{e,f'\}\subseteq E(G)$, and $\{e',f''\}\subseteq E(\bar{G})$, 
then $G-e+e'-f'+f'' = G-\{e,f'\}+\{e',f''\}$. Thus, for every edge $e\in E(G)$, there exists $f' \in E(G)$ and $\{e',f''\}\subseteq E(\bar{G})$ so that if the former edges are replaced with the latter, we get a graph isomorphic to the original. I.e. $G$ is $2$-swappable. 
\end{proof}

\begin{remark}
If $G-e\cong G-f$, there exist 2 distinct edges of $H$ that make suitable choices for $f'$. We choose the one that is not $e'$.
\end{remark}

Indeed, this makes a fair bit of sense. One would expect graphs with high edge reconstruction numbers and graphs that are finitely swappable to both have a relatively high level of symmetry. And so it turned out that 2-swappability was a necessary condition for a graph to have an edge reconstruction number higher than two. Then, what could be said of sufficient conditions? For this, we examined graphs with the highest level of symmetry. Namely, those for which all edges are removal similar. (i.e. All of their edge cards are isomorphic.) In regular graphs, that line of questioning gave us the following result.

\begin{theorem}
If $G$ is a regular, 2-swappable graph for which all edges are removal similar, then $ern(G)\geq 3$.
\end{theorem}

\begin{proof}
Let $G$ be a regular, 2-swappable graph, for which all edges are removal similar. Let $e\in E(G)$. Then there exist  $f\in E(G-e)$ and $e',f'\in E(\bar{G})$
so that $$G-e+e'-f+f'\cong G$$.
Let $H=G-e+e'$. Since $G$ is regular, $H\not\cong G$. However, $H-e'=G-e$. Further, $H-f\cong G-\hat{f} \cong G-e$ for some $\hat{f}\in G$.
Thus, $\left\{H-e', H-f\right\} =\left\{G-e,G-e\right\}\subseteq \mathcal{ED}(H)$.

Let $S\subseteq \mathcal{ED}(G)$ with $\|S\|=2$. Since all edges of $G$ are removal similar,
$$S=\left\{G-e,G-e\right\}\subseteq \mathcal{ED}(H).$$
Therefore, $G$ is not reconstructable from $S$. So, $ern(G)\geq 3$.
\end{proof}

Finally we arrived at the most important question: "Do such graphs even exist?" If so, the conjecture we were examining would be false. But as 2-swappable graphs had never been studied before, it was not known; that is, until we discovered a wonderfully simple example, the standard 3 dimensional cube.


\begin{figure}[h]
\centering
\begin{subfigure}[b]{.24\textwidth}
\includegraphics[width=\textwidth]{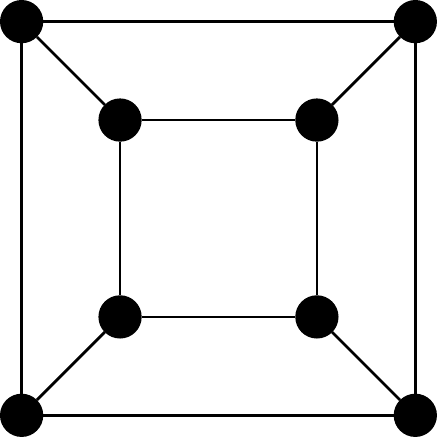}
\caption{$Q_3$}
\end{subfigure}
~
\begin{subfigure}[b]{.24\textwidth}
\includegraphics[width=\textwidth]{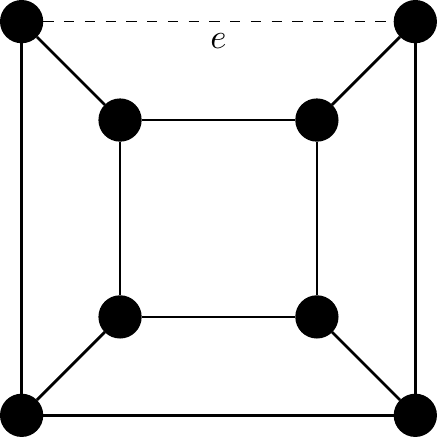}
\caption{$Q_3-e$}
\end{subfigure}
~
\begin{subfigure}[b]{.24\textwidth}
\includegraphics[width=\textwidth]{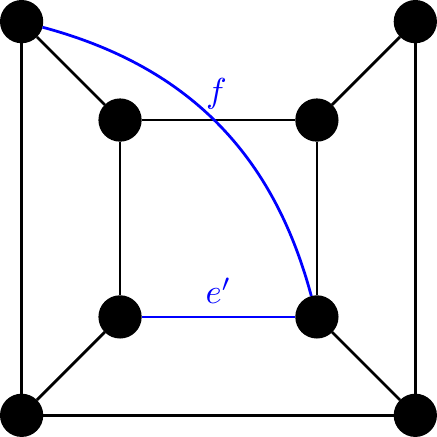}
\caption{$Q_3-e+f$}
\end{subfigure}
~
\begin{subfigure}[b]{.24\textwidth}
\includegraphics[width=\textwidth]{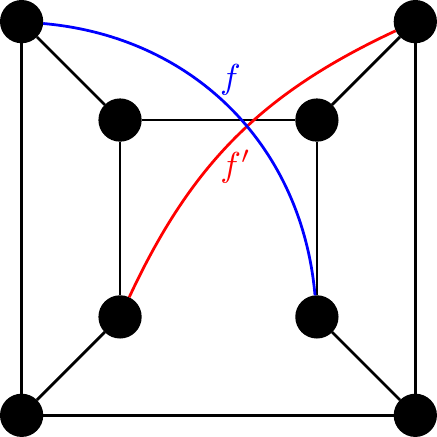}
\caption{$Q_3-e-e'+f+f'$}
\end{subfigure}
\caption{$Q_3$ is 2-swappable.}
\end{figure}


This counterexample quickly inspired the discovery of four infinite families of regular graphs that all had an edge reconstruction number greater than two, listed and proven below.

\begin{theorem}
If $G \cong K_n-M$, where $M$ is a perfect matching, then $G$ is 2-swappable.
\end{theorem}

\begin{proof}
Let $\{u,v\}\in E(G)$. Then there exist $u',v' \in G$ such that $\{u,u'\},\{v,v'\} \notin E(G)$, and $\{u',v'\} \in E(G)$. Then,
$$ G \cong G-\{u,v\} -\{u',v'\} + \{u,u'\} +\{v,v'\} $$
by a near-identity mapping that swaps $u$ with $v'$ and maps all other vertices to themselves. Thus, $G$ is 2-swappable.
\end{proof}

\begin{theorem}
If $G \cong K_n-H$ where $H$ is a Hamiltonian cycle, and $n\geq 5$ then $G$ is 2-swappable.
\end{theorem}

\begin{proof}
Let $\{u,v\}\in E(G)$ then $\{u,v\}$ bisects $H$ into two edge disjoint paths, $H_1$ and $H_2$. Without loss of generality, let $\{u,u'\}\in H_1$ and $\{v,v'\} \in H_2$. Then, $$H':=\{H_1-\{u,u'\}+\{u,v\}\}\cup\{H_2-\{v,v'\}+\{u',v'\}\}$$ is a Hamiltonian cycle in $K_n$. Therefore,
$$ G-\{u,v\} -\{u',v'\} + \{u,u'\} +\{v,v'\} = \bar{H'} \cong \bar{H} = G $$.

Thus, $G$ is 2-swappable.

\end{proof}

\begin{definition}
Given a bipartite graph $G_{n,m}$, the bipartite complement of $G$, is $\hat{G}:=K_{n,m}-G$.
\end{definition}

\begin{theorem}
If $G\cong K_{n,n}-M$, where $M$ is a perfect matching, then $G$ is 2-swappable.
\end{theorem}

\begin{proof}
Let $\left\{u,v\right\}\in E(G)$. Then there exist $u',v'\in G$ such that $\{u,u'\}, \{v,v'\}\in E(\hat{G})$ and $\left\{u',v'\right\}\in E(G)$.
Then, $$ G \cong G-\{u,v\} -\{u',v'\} + \{u,u'\} +\{v,v'\} $$ by a near-identity mapping that swaps $u$ with $v'$, (Which are in the same bipartite set) and maps all other vertices to themselves. Thus, $G$ is 2-swappable
\end{proof}

\begin{theorem}
If $G \cong K_{n,n} - H$, where $H$ is a Hamiltonian cycle, and $n\geq 4$ then $G$ is 2-swappable.
\end{theorem}

\begin{proof}
Let $\{u,v\} \in E(G)$. Then $\{u,v\}$ bisects $H$ into two edge disjoint paths, $H_1$ and $H_2$. Without loss of generality, let $\{u,u'\}\in H_1$ and $\{v,v'\} \in H_2$. Then, $$H':=\{H_1-\{u,u'\}+\{u,v\}\}\cup\{H_2-\{v,v'\}+\{u',v'\}\}$$ is a Hamiltonian cycle in $K_{n,n}$. Therefore,
$$ G-\{u,v\} -\{u',v'\} + \{u,u'\} +\{v,v'\} = \hat{H'} \cong \hat{H} = G $$.

Thus, $G$ is 2-swappable.
\end{proof}

\begin{theorem}
For each of the following graphs $G$, $ern(G)\geq 3$.
\begin{enumerate}
\item $G \cong K_n-M$, where $M$ is a perfect matching,
\item	$G \cong K_n-H$ where $H$ is a Hamiltonian cycle,
\item	$G\cong K_{n,n}-M$, where $M$ is a perfect matching,
\item	$G \cong K_{n,n} - H$, where $H$ is a Hamiltonian cycle.
\end{enumerate}
\end{theorem}

\begin{proof}
All of the above graphs are regular and 2-swappable. They are also all edge transitive, and thus all of their edges are respectively removal similar.
\end{proof}

Thus, we have not only demonstrated that there are indeed $r\geq3$ regular graphs with an edge reconstruction number that is greater than 2, but also done much of the ground work linking the topics of edge reconstruction and the swappability of graphs. The hope is that these results can be leveraged to provide insight into both areas in the future.

\section*{Future Work}
The framework of swappability as applied to edge reconstruction provides a number of potential leads for getting more edge reconstruction results. And so, we present the following questions and ideas.

\begin{question}
Do there exist other regular, 2-swappable graphs?
\end{question}

Currently, the only regular, connected, 2-swappable graphs known are cycles, and those 4 families presented in this paper. All of these are edge transitive, and thus all have an edge reconstruction number greater than 2.

\begin{question}
What are the actual edge reconstruction numbers of the graphs in this paper?
\end{question}

All we have done here is give a lower bound for the edge reconstruction number of these graphs. But what is their $ern$ exactly? It is known that this number is bounded above as a function of the regularity of the graph ($ern(G)\leq r-2$). But since our results refuted the conjecure that this number was bounded above by 2, the following questions must be asked.

\begin{question}
Let $R:=\{G|G \textrm{ is an } r\geq3 \textrm{ regular graph}\}$. Does there exist $k$ such that for all $G\in R$, $ern(G)\leq k$?
\end{question}

\begin{question}
Given $r\geq 3$, do there exist $r$-regular graphs $G_3,G_4,\dots,G_{r-2}$ such that $ern(G_i)=i$?
\end{question}

Lastly, 2-swappability implying $ern(G)\geq 3$ has only been proven for when $G$ has only one type of card in its edge-deck. Thus we have the last question:

\begin{question}
Can the notion of 2-swappability be modified to provide a necessary and sufficient condition for $ern(G)\geq 3$?
\end{question}

We believe this to be possible, and it would be something along the lines of for every pairing of edge-cards, it must be possible to swap 2 edges of those respective types with 2 edges from the complement. Giving a sort of "Full 2-swappability". We have played with this idea a bit, and the results seem promising.

\end{document}